\documentclass[a4paper,12pt,reqno]{amsart}
\usepackage{mathrsfs,latexsym, bm, amscd,amssymb,url}
\usepackage{xcolor, mathtools}
\PassOptionsToPackage{pdfusetitle,pagebackref,colorlinks}{hyperref}
\usepackage{bookmark}
\hypersetup{
  linkcolor={red!70!black},
  citecolor={green!70!black},
  urlcolor={teal!80!black},
	linkcolor={teal!80!black}
}

\usepackage[all,cmtip]{xy}  
\usepackage{graphicx}
\usepackage{color}
\usepackage{hyperref}
\usepackage{enumerate}
\usepackage{stmaryrd}

\usepackage[T1]{fontenc}         
\usepackage{ae}									
\usepackage[all,cmtip]{xy}  
\usepackage{graphicx}
\usepackage{color}

\newtheorem{theorem}{Theorem}[section]

\newtheorem{lemma}[theorem]{Lemma}

\newtheorem{corollary}[theorem]{Corollary}

\newtheorem{conjecture}[theorem]{Conjecture}

\theoremstyle{question}

\theoremstyle{definition}

\theoremstyle{remark}
\newtheorem{remark}[theorem]{Remark}

\newtheoremstyle{cited}{.5\baselineskip\@plus.2\baselineskip\@minus.2\baselineskip}{.5\baselineskip\@plus.2\baselineskip\@minus.2\baselineskip}{\itshape}{}{\bfseries}{\bfseries .}{5pt plus 1pt minus 1pt}{\thmname{#1}\thmnumber{~#2}\thmnote{ \normalfont#3}}
\theoremstyle{cited}

\newtheoremstyle{citeddef}{.5\baselineskip\@plus.2\baselineskip\@minus.2\baselineskip}{.5\baselineskip\@plus.2\baselineskip\@minus.2\baselineskip}{}{}{\bfseries}{\bfseries .}{5pt plus 1pt minus 1pt}{\thmname{#1}\thmnumber{~#2}\thmnote{ \normalfont#3}}
\theoremstyle{citeddef}

\newcommand{\Q}{\mathbb Q}

\newcommand{\Sing}{\operatorname{Sing}}

\newcommand{\codim}{\operatorname{codim}}

\newcommand{\supp}{\operatorname{supp}}

\newcommand{\dashedlongrightarrow}{\xymatrix@1@=15pt{\ar@{-->}[r]&}}
\renewcommand{\longrightarrow}{\xymatrix@1@=15pt{\ar[r]&}}
\renewcommand{\mapsto}{\xymatrix@1@=15pt{\ar@{|->}[r]&}}
\renewcommand{\twoheadrightarrow}{\xymatrix@1@=15pt{\ar@{->>}[r]&}}
\newcommand{\hooklongrightarrow}{\xymatrix@1@=15pt{\ar@{^(->}[r]&}}
\newcommand{\congpf}{\xymatrix@1@=15pt{\ar[r]^-\sim&}}
\renewcommand{\cong}{\simeq}

\usepackage{tikz-cd}
\usepackage{enumitem}

\PassOptionsToPackage{pdfusetitle,pagebackref,colorlinks}{hyperref}
\usepackage{bookmark}
\hypersetup{
  linkcolor={red!70!black},
  citecolor={green!70!black},
  urlcolor={blue!80!black}
}

\newcommand{\onto}{\rightarrow\hspace*{-.14in}\rightarrow}

\begin{document}  
\title[One-forms on varieties of Kodaira codimension one]{Nowhere vanishing holomorphic one-forms on varieties of Kodaira codimension one}

\author{Feng Hao}

\address{Department of Mathematics, KU Leuven, Celestijnenlaan 200b-box 2400, 3001 Leuven, Belgium.}
\email{feng.hao@kuleuven.be}
\date{\today}
%
%


\begin{abstract} 
Based on the celebrated result on zeros of holomorphic 1-forms on complex varieties of general type by Popa and Schnell, we study holomorphic 1-forms on $n$-dimensional varieties of Kodaira dimension $n-1$. We show that a  complex minimal smooth projective variety $X$ of Kodaira dimension $\kappa(X)=\dim X-1$ admits a holomorphic 1-form without zero if and only if there is a smooth morphism from $X$ to an elliptic curve.  Furthermore, for a general smooth projective variety (not necessarily minimal) $X$ of Kodaira codimension one, we give a structure theorem for $X$ given that $X$ admits a holomorphic 1-form without zero.
\end{abstract}

\maketitle

\section{Introduction}\label{intro}

Popa and Schnell \cite{PS14} showed that any holomorphic 1-form on a complex smooth projective variety of general type vanishes at some point. This indicates that the existence of nowhere vanishing holomorphic 1-forms encodes much algebro-geometric information of irregular smooth complex projective varieties. Also, one can refer to results, e.g., \cite{GL87}, \cite{Ca74}, \cite{Zh97}, \cite{LZ05}, \cite{HK05}, prior to \cite{PS14} in this direction. From another point of view, it was observed by Kotschick \cite{Ko22} and Schreieder \cite{Sch21} that the existence of nowhere vanishing holomorphic 1-forms on smooth projective varieties has a strong restriction on the topology of varieties (see e.g., \cite{HS21(1)}, \cite{DHL21}, \cite{SY22} for more results in this direction). We will focus on the algebro-geometric aspects on Kodaira codimension one varieties with  nowhere vanishing holomorphic 1-forms. Also, all varieties in this article are defined over the field of complex numbers.

First of all, we have the following theorem which is a generalization of results for surface case in \cite{Sch21} and threefold case in \cite{HS21(1)}. It can also be regarded as the next step after the result of \cite{HK05}, which shows that any minimal smooth projective variety of general type does not admit a nowhere vanishing holomorphic 1-form.

\begin{theorem}\label{Theorem:main A}
	Let $X$ be a minimal smooth projective variety with Kodaira dimension $\kappa(X)=\dim X-1$. Then $X$ admits a nowhere vanishing holomorphic 1-form $\omega\in H^0(X, \Omega_X^1)$ if and only if $X$ admits a smooth morphism to an elliptic curve. 
\end{theorem}

Theorem \ref{Theorem:main A} is the special case of the following theorem stated for morphisms from minimal smooth projective varieties of Kodaira codimension one to abelian varieties.
\begin{theorem}[=Theorem \ref{Theorem:main3}]\label{Theorem:main B}
	Let $X$ be a minimal smooth  projective variety with Kodaira dimension $\kappa(X)=\dim X-1$, and $f\colon X\to A$ be a morphism to an abelian variety $A$. Then the following are equivalent
	
	(1) There exists a holomorphic 1-form $\omega\in H^0(A, \Omega_A^1)$ such that $f^*\omega$ has no zero on $X$.
	
	(2) $X$ admits a smooth morphism $\varphi\colon X\to E_0$, where $E_0$ is an elliptic curve, such that $\varphi$ fits into the commutative diagram
	$$\xymatrix{
		X  \ar[d]_-{\varphi} \ar[r]^f&A \ar[ld]^-q\\
		E_0
	}$$
	where $q$ is a surjective morphism. Moreover, there is a finite \'etale covering $\tau\colon X' \to X$ such that $X'\cong Z\times E$, where $Z$ is a smooth minimal model of general type, and  for any closen point $z\in Z$ the following composition $$\{z\}\times E\hookrightarrow Z\times E\overset{\tau}{\longrightarrow} X \overset{\varphi}{\longrightarrow} E_0$$ is an isogeny.
\end{theorem}

For the general (not necessarily minimal) case, we have a structure theorem for smooth projective variety of Kodaira codimension one with a nowhere vanishing holomorphic 1-form.

\begin{theorem}[=Theorem \ref{Theorem:main2}]\label{Theorem:mainC}
	Let $X$ be a smooth projective variety of Kodaira dimension $\kappa(X)=\dim X-1$. If $X$ has a nowhere vanishing holomorphic 1-form, then for any minimal model $X^{\text{min}}$ of $X$ there is a finite quasi-\'etale covering $X'\to X^{\text{min}}$ such that any $\Q$-factorialization $X''$ of $X'$ is a product $Z\times E$, where $Z$ is a minimal model of general type and $E$ is an elliptic curve.
\end{theorem}

For a smooth projective variety of Kodaira dimension $\kappa(X)=\dim X-1$, $X$ admits a minimal model $X^{\text{min}}$ by \cite{BCHM10} and \cite{Lai11}. We briefly recall ``quasi-\'etale covering'' and ``$\mathbb{Q}$-factorialization'' in Section \ref{Section:2}. The proof of Theorem \ref{Theorem:mainC} is based on Popa and Schnell\cite[Theorem 2.1]{PS14}, Grassi and Wen\cite[Theorem 40]{GW19} on the birational modification of elliptic fibrations and constructing tricks in \cite[Section 5]{HS21(1)}. Theorem \ref{Theorem:main B} is derived from Theorem \ref{Theorem:mainC} via the specialty of flops of minimal varieties of type $Z\times E$, where $E$ is an elliptic curve. A direct consequence of Theorem \ref{Theorem:mainC} is

 \begin{corollary}\label{Corollary:mainA}
	Let $X$ be a smooth projective variety of Kodaira dimension $\kappa(X)=\dim X-1$. Assume that the Albanese $A_X$ of $X$ does not admit a 1-dimensional simple factor, then every holomorphic 1-form of $X$ has zero.
\end{corollary}

In the remaining of the introduction, we consider the following conjecture which is mentioned implicitly in \cite{DHL21} and supported by main results in \cite{DHL21} and \cite{SY22}.

\begin{conjecture}
Let $f\colon X\to A$ be a morphism from a smooth projective variety $X$ to a simple abelian variety $A$. $f$ is smooth if and only if there is a holomorphic 1-form $\omega\in H^0(A, \Omega_A^1)$ such that $f^*\omega$ has no zero. 
\end{conjecture}

The above conjecture is vacuously true for varieties of general type due to \cite{PS14}. We prove this conjecture for smooth projective varieties of Kodaira codimension one.

\begin{theorem}\label{Theorem:mainD}
Let $X$ be a smooth projective variety of Kodaira dimension $\kappa(X)=\dim X-1$, and $f\colon X\to A$ be a morphism to a simple abelian variety $A$. Then 

1) $f$ is smooth if and only if there is a holomorphic 1-form $\omega\in H^0(A, \Omega^1_A)$ such that $f^*\omega$ has no zero, and

2) $A$ is an elliptic curve when the conditions in 1) hold.
\end{theorem}

\noindent\textbf{Acknowledgements} Most of the work was carried out when the author visited Leibniz University Hannover. He would like to thank Stefan Schreieder for very helpful discussions and Leibniz University Hannover for  hospitality during the visit. The author also thanks Yajnaseni Dutta for her useful comments. This work is supported by the Research Foundation Flanders (FWO) Grant no. 1280421N “Topology, birational geometry and vanishing theorem for complex algebraic varieties”.

\section{Preliminaries and Technical Tools}\label{Section:2}

\subsection{Iitaka fibration}\label{sect:Iitaka-fib} We recall the following basic concepts on Iitaka and Kodaira dimension (see e.g., \cite[Section 2.1 C]{Laz04}). Let $X$ be a normal projective variety. Consider a line bundle $L$ on $X$ such that $H^0(X, L^m)\neq 0$ for some integer $m\in\mathbb{N}$. Then one has a rational mapping $$\phi_{|L^{\otimes m}|}\colon X\dashedlongrightarrow \mathbb{P}H^0(X, L^{\otimes m}).$$ The \emph{Iitaka dimension} of $L$ is defined to be $$\kappa(X ,L)\coloneqq \text{max}\{\dim\phi_{|L^{\otimes m}|}\ |\ m\in \mathbb{N}\ \text{such that}\ H^0(X,  L^{\otimes m})\neq 0\}.$$ By convention $\kappa(X ,L)=-\infty$ if $H^0(X, L^m)=0$ for all $m$. By \cite[Theorem 2.1.33]{Laz04}, for the fixed line bundle $L$ with $\kappa(X ,L)\geq 0$ and sufficiently divisable integers $m$, the rational maps $\phi_{|L^{\otimes m}|}$ are birationally equivalent to a fixed algebraic fibre space. We call any rational map in this birational class an \emph{Iitaka fibration} of $X$ associated to $L$. When the line bundle $L$ is the canonical line bundle $\mathcal{O}_X(K_X)$, we call $\kappa(X)\coloneqq \kappa(X, \mathcal{O}_X(K_X))$ the \emph{Kodaira dimension}
of $X$, and $\phi_{|L^{\otimes m}|}$ an \emph{Iitaka fibration} of $X$ for $m$ sufficiently divisable.

\subsection{Quasi-\'etale morphism} A finite morphism $f:X'\to X$ between normal varieties is called quasi-\'etale if it is \'etale in codimension one, see e.g.\ \cite{GKP16}; if $X$ is smooth, then any quasi-\'etale morphism $f\colon X'\to X$ is \'etale.
In particular, quasi-\'etale morphism $f$ is ramified at most at singular points of $X$.
\subsection{$\mathbb{Q}$-factorialization} Let $X$ be a terminal minimal variety with canonical divisor $K_X$ a $\Q$-cartier divisor. By \cite[Corollary 1.4.3]{BCHM10}, there is a $\Q$-factorialization (not necessarily unique) $\sigma: X'\to X$, i.e.\ a proper birational morphism which is an isomorphism in codimension one such that $X'$ is $\Q$-factorial, terminal and $K_{X'}$ is nef. 

\subsection{Popa-Schnell's result on  holomorphic 1-forms} We recall the following celebrated theorem by Popa and Schnell \cite[Theorem 2.1]{PS14}, which helps us to reduce our arguments for main theorems to cases of generically isotrivial Iitaka fibrations. 

\begin{theorem}[Popa-Schnell]\label{thm:popa-schnell}
	Let $X$ be a smooth projective variety, and $f\colon X\to A$
	be a morphism to an abelian variety. If $H^0(X, \mathcal{O}_X(mK_X- f^*L))\neq 0$ for some integer $m\geq 1$
	and some ample divisor $L$ on $A$, then $Z(\omega)$ is nonempty for every $\omega$ in the image of the map $f^*\colon H^0(A, \Omega_A^1)\to H^0(X, \Omega_X^1)$.
\end{theorem}

\subsection{Birational modification of elliptic fibration} In  \cite[Theorem 40]{GW19}, Grassi and Wen associate an elliptic fibration with a birational model that is easier to deal with. 

\begin{theorem}[Grassi-Wen]\label{Thm:GW-ellipt-modif}
	Let $\phi\colon X\to S$ be an elliptic fibration such that $X$ has $\mathbb{Q}$-factorial terminal
	singularities, $S$ is normal, and the canonical divisor $K_X=\phi^*L$ where $L$ is a $\mathbb{Q}$-Cartier divisor on $S$. Then one has the following commutative diagram
	$$
	\xymatrix{
		X \ar[d]_{\phi} \ar@{-->}[r]^-{\alpha} &Y\ar[ld]_-{\phi'}\ar[d]^-{\psi} \\
		S  & T\ar[l]^{\beta} ,
	}
	$$
	where $\alpha$ is a birational map, $\beta$ is a birational morphism, and $\psi$ is an elliptic fibration  together with an effective $\mathbb{Q}$-divisor $\Lambda_T$ such that:
	
	\begin{enumerate}
		\item \label{item:GW-ellipt-modif(1)} $Y$ has $\mathbb{Q}$-factorial terminal singularities,
		
		\item \label{item:GW-ellipt-modif(2)} $K_Y=\psi^*(K_T+\Lambda_T)={\phi'}^*L$ where $(T,\Lambda_T)$ is klt,
		
		\item \label{item:GW-ellipt-modif(3)} there is no effective divisor $E$ in $Y$ such that $\codim \psi(E)\geq2$.
	\end{enumerate}
\end{theorem}

\begin{remark}
	The above theorem is a higher dimensional generalization of \cite[Theorem A.1]{Na02} for threefolds. All the above claims are stated in  \cite[Theorem 40]{GW19} except for $K_Y=\phi'^*L$, which in fact follows from the proof.  For the convenience of the reader, we provide the full proof for the above theorem.
\end{remark}

\begin{proof}
We proceed by the induction on the relative Picard number $\rho(X/S)$. Suppose that there is an integral divisor $E$ on $X$ such that $\codim \phi(E)\geq 2$. Then one can choose a Cartier divisor $C\subset S$ containing $\phi(E)$, such that the base locus $Bs(|C|)$ is of codimension at least 2, and $\phi^*C=D+F$ with $F$ the maximal component so that $\codim \phi(F)\geq 2$ and $D=\phi^*C-F$. Thus one has $\supp(E)\subset \supp(F)$ and $\codim \phi(D)=1$. Since $K_X=\phi^*L$, $K_X+D$ is $\phi$-nef if and only if $D$ is $\phi$-nef.

In the case that $D$ is not $\phi$-nef, we consider the terminal log pair $(X, \epsilon D)$ for $0<\epsilon\ll 1$. Run the relative MMP for $(X, \epsilon D)$, we get $(X_1, \epsilon D_1)$ over $S$. Notice that the base locus of the linear system $|D|$ is of $\codim \geq 2$ in $X$ by the choice of $C$. If $l$ is a curve contracted in the step $\alpha_1\colon (X, \epsilon D)\dashedlongrightarrow (X_1, \epsilon D_1)$, one has $l\cdot D<0$ since $K_X$ is trivial over $S$ and $l$ is a $K_X+\epsilon D$-negative curve. Hence the contracted curves of $\alpha_1$ are contained in $D\cap \phi^{-1}(Bs(|C|))$, which is of codimension 2 in $X$. Thus $\alpha_1$ is a $D$-flop.

Runing the above program, we get a sequence of  flops and arrive at a birational model $(Y, \epsilon D')$ over $S$ 
	$$
\xymatrix{
	(X, \epsilon D) \ar[d]_-\phi \ar@{-->}[r]^-\alpha &(Y, \epsilon D')\ar[ld]^-{\phi'} \\
	S,
}
$$
so that $D'=\alpha_* D$ is nef over $S$,  $K_Y=\phi'^*L$,  and $Y$ has $\mathbb{Q}$-factorial terminal singularities.

Denote $F'\coloneqq \phi'^*C-D'$. Since $\codim \phi(F)\geq2$, $\codim \phi'(F')\geq2$. Also, $-F'=-\phi'^*C+D'$ is $\phi'$-semiample by \cite[Theorem A.4]{Na02}, then for a large integer $m$ the local system $|-mF'|$ on $Y$ gives a morphism $\psi\colon Y\to T$ and a morphism $\beta\colon T\to S$, i.e., a commutative diagram 	$$
\xymatrix{
	X \ar[d]_{\phi} \ar@{-->}[r]^-{\alpha} &Y\ar[ld]_-{\phi'}\ar[d]^-{\psi} \\
	S  & T\ar[l]^{\beta}.
}
$$ Note that $\beta$ is not an isomorphism, since $-F'=-\phi'^*C+D'$ is nummerically trivial over $T$ but not $S$ (Note that there exist $D'$-negative flopping curves on $Y$ contracted by $\phi'$). Note also $F'\cdot R=0$ for general fibres $R$ of $\phi'$. Hence $\psi$ is an elliptic fibration and $\beta$ is birational. Moreover, $K_Y=\psi^*(\beta^*L)$, thus there is an effective $\Q$-divisor $\Lambda_T$ on $T$ so that $(T, \Lambda_T)$ is klt and $K_Y=\psi^*(K_T+\Lambda_T)$ by \cite[Theorem 0.4]{Na87}. In the end, we notice that $\rho(X/S)>\rho(Y/T))$, since $\codim \phi(F)\geq2$ and $F'=-\psi^*A$ for some $\beta$-ample $\Q$-divisor $A$ on $T$. Hence by the induction on relative Picard numbers, we get a birational model $\psi\colon Y\to T$ satisfying $(1), (2), (3)$. 
\end{proof} 
 
\section{Proof of Main Theorems}

For an $n$-dimensional smooth projective variety of Kodaira dimension $\kappa(X)=n-1$, $X$ has a good minimal model $X^{\text{min}}$ together with a birational map $\tau: X\dashrightarrow X^{\text{min}}$ by \cite[Theorem 4.4]{Lai11}. Then the linear system $|mK_{X^{\text{min}}}|$ is base point free for a sufficiently divisable integer $m$. Hence the Iitaka fibration $\phi\coloneqq \phi_{|{mK_{X^{\text{min}}}}|}\colon X^{\text{min}}\to S$ is a morphism, whose general fibres are elliptic curves. 
Since $X^{\text{min}}$ has rational singularities, one has the commutative diagram 	$$
\xymatrix{
	X\ar@{-->}[d]_{\tau}\ar[rrd]^-{f}\\
	X^{\text{min}}\ar[rr]^-{f^{\text{min}}}&& A.}
$$
With the above notations and applying Theorem \ref{thm:popa-schnell}, we have

\begin{lemma}\label{prop:reduce to isotrivial}
	Let $X$ be an $n$-dimensional smooth projective variety with Kodaira dimension $\kappa(X)=n-1$, and $f\colon X\to A$ be a morphism to an abelian variety $A$. Assume that there exists a holomorphic 1-form $\omega\in H^0(A, \Omega_A^1)$ such that $f^*\omega$ has no zero on $X$, then  
	
	1) $f^{\text{min}}$ does not contract the fibres of Iitaka fibrations $\phi \colon X^{\text{min}}\to S$  of $X$, and 
	
	2) Iitaka fibrations of $X$ are generically isotrivial, i.e., general fibres of $\phi$ are isomorphic to each other.
\end{lemma}

\begin{proof}
 Note first 1) implies 2). In fact, since $f^{\text{min}}$ does not contract the fibres of an Iitaka fibration $\phi \colon X^{\text{min}}\to S$, the fibres of $\phi$  map to translates of a fixed elliptic curve in $A$ via $f^{\text{min}}$. Also, note that $A$ only contains at most countably many abelian subvarieties. Hence general fibres of $\phi$ are isomorphic to each other.

Now we show statement 1). Suppose by contradiction that $f^{\text{min}}$  contracts the fibres of Iitaka fibrations $\phi \colon X^{\text{min}}\to S$. Since $S$ has klt singularities, in particular, $S$ has rational singularities, $f^{\text{min}}$ factors through $\phi$ together with a morphism $g:S\to A$. Hence we have the following commutative diagram 
	$$
	\xymatrix{
		X\ar@{-->}[d]_-{\tau}\ar[rrrrdd]^-{f}\\
		X^{\text{min}}\ar[d]_-{\phi}\ar[drr]_-{\mu\circ\phi}\ar[drrrr]^-{f^{\text{min}}}\\
		S \ar[rr]_-{\mu}&&S'\ar[rr]_-{g'} && A}
	$$
	where $\mu$ is the Stein factorization of $g$, $S'$ is normal and $g=g'\circ \mu$. Also, it is clear that the composition $\mu\circ\phi$ is the Stein factorization of $f^{\text{min}}$ and $(\mu\circ\phi)_*\mathcal{O}_{X^{\text{min}}}=\mathcal{O}_{S'}$. 
	By \cite[Proposition 1.14]{Mo87} and the proof of \cite[Definition-Theorem 1.1 (i)]{Mo87}, for any ample line bundle $H$ on $S'$, there is an integer $k\in\mathbb{N}$ such that 	$$H^0(X, \mathcal{O}_{X^{\text{min}}}(kK_{X^{\text{min}}}-(\mu\circ\phi)^*H))\neq 0.$$ Now we choose an ample  line bundle $L$ on $A$, then we get for some $k$
	$$H^0(X, \mathcal{O}_{X^{\text{min}}}(kK_{X^{\text{min}}}-{f^{\text{min}}}^*L))\neq 0,$$ since $g'$ is a finite morphism. Notice that $\tau$ does not extract any divisor, we then have for another integer $k'$
	$$H^0(X, \mathcal{O}_{X}(kk'K_{X}-f^*k'L))\neq 0.$$
	 Then by Theorem \ref{thm:popa-schnell} we get every holomorphic 1-form $\omega$ on $X$ has zero. This is a contradiction. 
\end{proof}

Also, we have the following lemma, which is a higher dimensional generalization of \cite[Lemma 5.5, Lemma 5.6]{HS21(1)}.

\begin{lemma} \label{lem:pi_1(X)}
	Let $\phi: X\rightarrow S$ be an elliptic fibration such that $X$ has $\mathbb{Q}$-factorial terminal singularities, $S$ is normal, and $K_X=\phi^*L$ with $L$ an effective $\Q$-Cartier divisor on $S$.
	Consider the birational morphism $\beta\colon T\to S$ and elliptic fibration $\psi :Y\to T$ in Theorem \ref{Thm:GW-ellipt-modif}.
	Then there is a smooth open subset $U\subset T$ such that 
	
	\begin{enumerate}
		\item \label{item:lem:pi_1(X)(1)} $\codim_T T\backslash U\geq 2$,
	
	\item \label{item:lem:pi_1(X)(2)} the preimage $Y_{U}\coloneqq{\psi}^{-1}(U)$ is nonsingular,
	\end{enumerate}
	
Moreover, for any open set $U$ satisfying (\ref{item:lem:pi_1(X)(1)}) and (\ref{item:lem:pi_1(X)(2)}), the birational map $Y_{U}\dashrightarrow X^{\text{sm}}$ induces an isomorphism
	$$
	\pi_1(Y_{U})\cong \pi_1(X^{\text{sm}}) ,
	$$ 
	where $X^{\text{sm}}\subset X$ denotes the smooth locus of $X$.
\end{lemma}
\begin{proof} By Theorem \ref{Thm:GW-ellipt-modif} (\ref{item:GW-ellipt-modif(1)}) (\ref{item:GW-ellipt-modif(2)}), the singular locus $\Sing Y$ is of codimension at least 3 in $Y$ (see e.g., \cite[Corollary 5.18]{KM08}) and $\Sing T$ is of codimension at least 2 in $T$. Then one can choose a smooth open subset $U\subset T$ such that $\codim_T T\backslash U\geq 2$ and $\psi^{-1}(U)$ is smooth.
	
We show the last claim for such an open set $U$.
	Since $K_Y$ is nef over $S$ by Theorem \ref{Thm:GW-ellipt-modif} (\ref{item:GW-ellipt-modif(2)}), $X$ and $Y$ are birational minimal models over $S$ and so they are  isomorphic in codimension one (see e.g.\ \cite[Theorem 3.52(2)]{KM08}).  By Theorem \ref{Thm:GW-ellipt-modif} (\ref{item:GW-ellipt-modif(3)}), $\dim Y\backslash Y_U\leq \dim Y-2$, since $\codim_T T\backslash U\geq2$.
Hence $Y_U$ and $X^{\text{sm}}$ are isomorphic in codimension one (notice that $\codim_X\Sing X\geq 3$).
Since $Y_U$ and $X^{\text{sm}}$ are smooth, one get $\pi_1(Y_{U})\cong \pi_1(X^{\text{sm}})$.
\end{proof}

Now we consider a smooth projective variety $X$ of dimension $n$ and Kodaira dimension $n-1$. Let $X^{\text{min}}$ be a minimal model of $X$. Consider an Iitaka fibration $$\phi\coloneqq \phi_{|mK_{X^{\text{min}}}|}\colon X^{\text{min}}\to S.$$ We have the following theorem.

\begin{theorem}\label{Theorem:main1} 
	Let $X$ be an $n$-dimensional smooth projective variety with Kodaira dimension $\kappa(X)=n-1$, and $f\colon X\to A$ be a morphism to an abelian variety $A$. Assume that there exists a holomorphic 1-form $\omega\in H^0(A, \Omega_A^1)$ such that $f^*\omega$ has no zero on $X$, then for any minimal model $X^{\text{min}}$ of $X$ there exists a finite quasi-\'etale covering $X'\to X^{\text{min}}$ such that 
	
	1) $X'$ is birational to $S'\times E$, where $E$ is an elliptic curve and $S'$ is a smooth projective variety with a generically finite rational map to the base $S$ of an Iitaka fibration of $X^{\text{min}}$, and
	
	2)  The second projection $p_2\colon X' \dashrightarrow E$ fits into the commutative diagram $$\xymatrix{
		X' \ar[r]\ar@{-->}[d]_-{p_2} &X^{\text{min}}\ar[d]^-{f^{\text{min}}} \\
		E\ar[d]^u &A\ar[ld]^q\\
	E_0
}$$ where $u$ is an isogeny between elliptic curves, and $q$ is a surjective morphism.
	
\end{theorem}

\begin{proof}
	By Theorem \ref{Thm:GW-ellipt-modif}, for the Iitaka fibration $\phi\coloneqq \phi_{|mK_{X^{\text{min}}}|}\colon X^{\text{min}}\to S$, there is a birational morphism $\beta\colon T\to S$ and an elliptic fibration $\psi\colon Y\to T$  that is birational to $\phi$ and satisfies (\ref{item:GW-ellipt-modif(1)}), (\ref{item:GW-ellipt-modif(2)}), (\ref{item:GW-ellipt-modif(3)}) of Theorem \ref{Thm:GW-ellipt-modif}.
	By Lemma \ref{lem:pi_1(X)}, there is a smooth open subset $U\subset T$ such that $\codim_TT\backslash U\geq2$, $Y_{U}\coloneqq{\psi}^{-1}(U)$ is smooth and the birational map $Y_{U}\dashrightarrow X^{\text{sm}}$ induces an isomorphism 
	\begin{align} \label{eq:pi_1(X)}
	\pi_1(Y_{U})\cong \pi_1(X^{\text{sm}}),
	\end{align} where $X^{\text{sm}}$ is the smooth loci of $X^{\text{min}}$. 	By Theorem \ref{Thm:GW-ellipt-modif} (\ref{item:GW-ellipt-modif(3)}), there exists an open subset $V\subset T$ with $\codim_TT\backslash V\geq3$ so that $\psi$ is equidimenisonal with one dimensional fibres over $V$. Hence we may assume in the beginning that $\psi|_U\colon Y_U\to U$ is equidimensional for the above chosen $U$.
	
By Theorem \ref{Thm:GW-ellipt-modif}, $X^{\text{min}}$ and $Y$ have rational singularities. Since $X$ admits a nowhere vanishing holomorphic 1-form, we have that fibres of $\phi\colon X^{\text{min}}\to S$ are not contracted by the induced map $f^{\text{min}}\colon {X^{\text{min}}}\to A$ according to Lemma \ref{prop:reduce to isotrivial}. Since $X^{\text{min}}$ is birational to $Y$, we also have the induced morphism $f'\colon Y\to A$. Hence fibres of $\psi:Y\to T$ are not contracted by $f'\colon Y\to A$.  
	The fibres of $\psi$ are mapped to translates of a fixed elliptic curve $E_0\subset A$ via $f'$.
	Since $A$ is an abelian variety, we can dualize this inclusion to get a surjection $A\to A^{\vee}\twoheadrightarrow E_0$.  
	Composing this morphism with $f'$, we get a surjection 
	$$
	p\colon Y\longrightarrow E_0 ,
	$$
	which restricts to finite \'etale covers on general fibres of $\psi$.
	Taking the Stein factorization, we may assume that $p$ has connected fibres. In fact, one can take the Stein factorization $q\colon A\to E_0$ of $A\to A^{\vee}\twoheadrightarrow E_0$ and compose it with $f'$. 

	Since $Y$ is terminal, $\codim_Y \Sing Y\geq3$.
	Hence $\codim_{\widetilde T}\Sing\widetilde T\geq 3$ for a general fibre $\widetilde T=p^{-1}(e)$ of $p$, by the generic smoothness theorem. Note that we have an induced generically finite morphism $$\psi|_{\widetilde T}\colon \widetilde T\to T,$$ and 
	$\psi|_{\widetilde T}$ is a finite morphism over $U$, because $\psi$ is equidimensional over $U$. Since $\codim_{\widetilde T}\Sing\widetilde T\geq 3$, we may assume that for the previously chosen $U$, $\psi|_{\widetilde T}^{-1}(U)$ is a smooth open subset in $\widetilde T$. Now we consider the normalization $\widetilde Y$ of the base change
$Y\times_{T}\widetilde T$ and the corresponding commutative diagram
	$$
	\xymatrix{ 
		\widetilde Y\ar[r]\ar[d]_{\widetilde \psi} &Y\ar[d]^{\psi}\\
		\widetilde T \ar[r]^{\psi|_{\widetilde T}} &T .} 
	$$
	We denote $\widetilde U\coloneqq 
	\psi|_{\widetilde T}^{-1}(U)\subset \widetilde T$ and consider the base change $\widetilde Y_{\widetilde U}=\widetilde \psi^{-1}(\widetilde U)\subset \widetilde Y$.
	Since additionally $K_{Y}$ is nef over $T$, the base change $Y_Z$ to a general complete intersection curve $Z\subset T$ is an isotrivial smooth minimal elliptic surface by Lemma \ref{prop:reduce to isotrivial} 2). Because the fibres of $\psi$ are mapped onto translates of a fixed elliptic curve $E_0$ in $A$, all the singular fibres are multiples of smooth elliptic curves (see e.g., \cite[p.\ 201]{BHPV04} for the classification of singular fibres).
    Thus all singular fibres of $\psi$ are multiples of smooth elliptic curves over codimension one point of $T$.  Now choose a general complete intersection curve $C\subset U$ and let $\widetilde C\subset \widetilde U$ be the preimage of $C$ in $\widetilde U$.
	Applying Lemma \cite[Lemma 5.11]{HS21(1)} to the base change of $\widetilde Y_{\widetilde U}$ and $Y_{U}$ to $\widetilde C$ and $C$, respectively, we find the following: up to removing a codimension two closed subset from $U$, we may assume that $\widetilde Y_{\widetilde U}\to \widetilde U$ is a smooth elliptic fibre bundle and $\widetilde Y_{\widetilde U}\to Y_{U}$ is \'etale.
	Since this bundle has a section by construction, the existence of a fine moduli space for elliptic curves with level structure shows that $\widetilde Y_{\widetilde U}\cong \widetilde U\times E$ for an elliptic curve $E$, which is isogeny to $E_0$.
	
	By \cite[Theorem 3.8]{GKP16}, any finite \'etale cover of $X^{\text{sm}}$ extends to a finite quasi-\'etale cover of $X^{\text{min}}$.
	Since
	$
	\pi_1(Y_{U})\cong \pi_1(X^{\text{sm}})
	$, the finite \'etale cover $\widetilde Y_{\widetilde U}\to Y_{U}$ is thus birational to a finite quasi-\'etale covering
	$$
	X'\to X^{\text{min}}
	$$
	of $X^{\text{min}}$.
	Since $\widetilde Y_{\widetilde U}\cong \widetilde U\times E$, we conclude that $X'$ is birational to $S'\times E$, where $S'$ is a smooth projective variety birational to $\widetilde U$ and $E$ is an elliptic curve.
	
Note that for general $s\in S'$, we have that the morphism $$\{s\}\times E\hookrightarrow X'\to X^{\text{min}}\overset{f^{\text{min}}}{\longrightarrow} A\overset{q}{\longrightarrow} E_0$$ is an isogeny. Hence 2) holds true.
\end{proof}	

\begin{theorem}\label{Theorem:main2}
Let $X$ be an $n$-dimensional smooth projective variety of Kodaira dimension $n-1$. If $X$ has a nowhere vanishing holomorphic 1-form, then for any minimal model $X^{\text{min}}$ of $X$ there is a finite quasi-\'etale covering $X'\to X^{\text{min}}$ such that any $\Q$-factorialization $X''$ of $X'$ is a product $X''\cong Z\times E$, where $Z$ is a minimal model of general type and $E$ is an elliptic curve.
\end{theorem}

\begin{proof}
 Fix a minimal model $X^{\text{min}}$ of $X$, applying Theorem \ref{Theorem:main1} to the Albanese map of $X$, we have that there exists a finite quasi-\'etale covering $X'\to X^{\text{min}}$ such that 
 $X'$ is birational to $S'\times E$, where $E$ is an elliptic curve and $S'$ is a smooth projective variety. 
 
 	Since $K_{X^{\text{min}}}$ is nef, so is $K_{X'}$.
Moreover, $X'$ is terminal by \cite[Proposition 5.20]{KM08}, because it is a finite quasi-\'etale cover. Also, it is clear that $\kappa(X')=\kappa(X)=n-1$. Thus $S'$ is of general type.
Let $S'_{\text{min}}$ be a minimal model of $S'$. Then $X''$ and $S'_{\text{min}}\times E$ are birational minimal models and so they are connected by a sequence of flops (see \cite{Ka08}).  Consider any flop of $S'_{\text{min}}\times E$ as in the following commutative diagram
$$
\xymatrix{ 
	S'_{\text{min}}\times E \ar[rd]_{a}\ar@{-->}[rr]&&  V\ar[ld]^{a^+}\\
	&W.} 
$$ 
Since any (rational) flopping curve in $S'_{\text{min}}\times E$ projects to a point on $E$, then the group action of $E$ can sweep up a trivial family of rational curves over $E$, which is contracted by the small contraction $a$. Hence $W=M\times E$ for some projective variety $M$, and $a$ induces a flopping contraction $a_0\colon S'_{\text{min}}\to M$. Therefore we get the flop $S''_{\text{min}}$ of $S'_{\text{min}}$ with respect to $a_0$, where $S''_{\text{min}}$ is another minimal variety which is birational to $S'_{\text{min}}$. By the uniqueness of flops, we have that $V\cong S''_{\text{min}}\times E$.
Hence $X''\cong Z\times E$, where $Z$ is a minimal model of general type. 
\end{proof}

When $X$ is an $n$-dimensional smooth minimal model with Kodaira dimension $\kappa(X)=n-1$. we have the following stronger theorem.

\begin{theorem}\label{Theorem:main3}
	Let $X$ be an $n$-dimensional minimal smooth  projective variety with Kodaira dimension $\kappa(X)=n-1$, and $f\colon X\to A$ be a morphism to an abelian variety $A$. Then the following are equivalent
	
	(1) There exists a holomorphic 1-form $\omega\in H^0(A, \Omega_A^1)$ such that $f^*\omega$ has no zero on $X$.
	
	(2) $X$ admits a smooth morphism $\varphi\colon X\to E_0$, where $E_0$ is an elliptic curve, such that $\varphi$ fits into the commutative diagram
	$$\xymatrix{
		X  \ar[d]_-{\varphi} \ar[r]^f&A \ar[ld]^-q\\
		E_0
	}$$
where $q$ is a surjective morphism. Moreover, there is a finite \'etale covering $\tau\colon X' \to X$ such that $X'\cong Z\times E$, where $Z$ is a smooth minimal model of general type, and  for any closed point $z\in Z$ the following composition $$\{z\}\times E\hookrightarrow Z\times E\overset{\tau}{\longrightarrow} X \overset{\varphi}{\longrightarrow} E_0$$ is an isogeny.
\end{theorem}
 
\begin{proof}
(2)$\Rightarrow$(1) is trivial. If (1) holds for $X$, by Theorem \ref{Theorem:main1}, we know that there is a  quasi-\'etale covering $\tau\colon X'\to X$ such that $X'$ is birational to $S'\times E$. Since $X$ is smooth and of Kodaira dimension $n-1$, $\tau$ is a finite \'etale covering by \cite[Corollary 2.4]{HS21(2)} and $S'$ is of general type. Then $S'$ admits a minimal model $S'_{\text{min}}$ by \cite{BCHM10}. Hence $S'_{\text{min}}\times E$ is a minimal model which is birational to the smooth minimal model $X'$. Therefore there is a sequence of flops connecting $X'$ and $S'_{\text{min}}\times E$. 
According to the argument in Theorem \ref{Theorem:main2}, we have $X'\cong Z\times E$, where $Z$ is a smooth minimal model of general type. 
 
Now for any closed point $z\in Z$, consider the following composition of morphisms $$\{z\}\times E\hookrightarrow Z\times E\overset{\tau}{\longrightarrow} X \overset{f}{\longrightarrow} A.$$ Since $f^*\omega$ has no zero on $X$, $\tau^*f^*\omega$ has no zero on $Z\times E$. We write $\tau^*f^*\omega=p_1^*\omega_1+p_2^*\omega_2$, where $\omega_1\in H^0(Z, \Omega_Z^1)$, $\omega_2\in H^0(E, \Omega_E^1)$ and $p_1$, $p_2$ are the natural projections. Since $Z$ is of general type, $\omega_1$ has zeros on $Z$ by \cite{PS14}. Therefore $\omega_2\neq 0$ and $\{z\}\times E$ map to translates of a fixed elliptic curve $E_0$ in $A$ for all $z\in Z$. We can take the dual morphism $A^{\vee}\onto E_0$ of the inclusion $E_0\hookrightarrow A$ and get the required smooth morphism $\varphi$ as the composition $$X\to A \to A ^{\vee}\to E_0.$$ Note that $\varphi$ is smooth, since $\varphi\circ\tau$ is smooth and $\tau$ is \'etale.
 
 \end{proof}

\begin{proof}[Proof of Theorem \ref{Corollary:mainA}]
Applying Theorem \ref{Theorem:main3} to the Albanese morphism $\alpha_X\colon X\to A_X$, we prove Theorem \ref{Corollary:mainA}.
\end{proof}
 
 \begin{proof}[Proof of Theorem \ref{Theorem:mainD}]
 	We apply Theorem \ref{Theorem:main1} to the morphism $f\colon X\to A$ with $A$ simple. Then for any minimal model $X^{\text{min}}$ of $X$ there exists a finite quasi-\'etale covering $X'\to X^{\text{min}}$ such that $X'$ is birational to $S'\times E$, where $E$ is an elliptic curve, and the second projection $p_2\colon X' \dashrightarrow E$ fits into the commutative diagram $$\xymatrix{
  	X' \ar[r]\ar@{-->}[d]_-{p_2} &X^{\text{min}}\ar[d]^-{f^{\text{min}}} \\
  	E\ar[d]^u &A\ar[ld]^q\\
  	E_0
  }$$ where $u$ is an isogeny between elliptic curves. Since $A$ is simple and $q$ is surjective, $A$ must be an elliptic curve. Also, since there is a holomorphic 1-form $\omega\in H^0(A, \Omega^1_A)$ such that $f^*\omega$ has no zero and $\dim A=1$, we have that $f$ is a submersion, i.e., $f$ is a smooth morphism.
  
 \end{proof}

Similar argument also proves Corollary \ref{Corollary:mainA}.

\begin{proof}[Proof of Corollary \ref{Corollary:mainA}]
Assume by contradiction that there exists a nowhere vanishing holomorphic 1-form on $X$, we have that there is a surjective morphism $q\colon A_X\to E_0$ from the Albanese variety $A_X$ of $X$ to an elliptic curve $E_0$ by Theorem \ref{Theorem:main1}. This contradicts the assumption in the corollary.
\end{proof}

\end{document}